\documentclass[11pt,leqno]{amsart}
\usepackage{epsfig}
\usepackage{amssymb}
\usepackage{amscd}
\usepackage[all,cmtip,matrix,arrow]{xy}
\usepackage{graphicx}
\xyoption{rotate}
\xyoption{pdf}
\usepackage{xparse}

\newtheorem{theo}{Theorem}[section]
\newtheorem{lemma}[theo]{Lemma}
\newtheorem{defi}[theo]{Definition}
\newtheorem{prop}[theo]{Proposition}

\newtheorem{conj}[theo]{Conjecture}
\newtheorem{cor}[theo]{Corollary}
\newtheorem{remark}[theo]{Remark}

\newtheorem{ques}[theo]{Question}
\numberwithin{equation}{section}

\mathchardef\mhyphen="2D

\def\A{{\mathbb A}}

\def\bL{\mathbb{L}}

\def\Z{\mathbb{Z}}

\def\bG{\mathbb{G}}

\def\wt{\widetilde}

\def\bR{{\mathbf R}}
\def\bL{{\mathbf L}}
\def\bY{{\mathbf Y}}

\def\pre-tr{\operatorname{pre-tr}}

\def\Hom{\operatorname{Hom}}
\def\End{\operatorname{End}}

\DeclareMathOperator*{\colim}{colim}

\newcommand{\Ltens}[1]{%
  \mathbin{\mathop{\otimes}\displaylimits^{\bL}_{#1}}%
}

\newcommand{\bbar}{\overline}

\newcommand{\xto}{\xrightarrow}

\newcommand{\hto}{\hookrightarrow}

\newcommand{\mk}{\mathrm k}

\newcommand{\cF}{{\mathcal F}}
\newcommand{\cG}{{\mathcal G}}
\newcommand{\cO}{{\mathcal O}}

\newcommand{\cA}{{\mathcal A}}
\newcommand{\cB}{{\mathcal B}}

\newcommand{\cC}{{\mathcal C}}

\newcommand{\cS}{{\mathcal S}}

\newcommand{\cK}{{\mathcal K}}

\newcommand{\veps}{\varepsilon}

\newcommand{\un}{\underline}

\newcommand{\Fun}{\operatorname{Fun}}

\newcommand{\Perf}{\operatorname{Perf}}
\newcommand{\PsPerf}{\operatorname{PsPerf}}
\newcommand{\perf}{\operatorname{perf}}
\newcommand{\pspe}{\operatorname{pspe}}

\newcommand{\ch}{\operatorname{ch}}

\newcommand{\Ext}{\operatorname{Ext}}

\newcommand{\str}{\operatorname{str}}

\newcommand{\Spec}{\operatorname{Spec}}

\newcommand{\Ho}{\operatorname{Ho}}

\newcommand{\id}{\operatorname{id}}

\newcommand{\dgalg}{\operatorname{dgalg}}
\newcommand{\dgcat}{\operatorname{dgcat}}

\newcommand{\Mod}{\operatorname{Mod}}

\usepackage{epsf}
\usepackage{amscd}

\newcommand{\bw}{\mathrm{w}}

\title[Categorical smooth compactifications and generalized degeneration]
{Categorical smooth compactifications and generalized Hodge-to-de Rham degeneration}

\author{Alexander I. Efimov}
\address{Steklov Mathematical Institute of RAS, Gubkin str. 8, GSP-1, Moscow 119991, Russia\\
National Research University Higher School of Economics, Russian Federation}
\email{efimov@mccme.ru}
\thanks{The author is partially supported by Laboratory of Mirror Symmetry NRU HSE, RF government  grant, ag. N 14.641.31.0001}

\begin{document}

\begin{abstract} We disprove two (unpublished) conjectures of Kontsevich which state generalized versions of categorical Hodge-to-de Rham degeneration for smooth and for proper DG categories (but not smooth and proper, in which case degeneration is proved by Kaledin \cite{Ka}). In particular, we show that there exists a minimal $10$-dimensional $A_{\infty}$-algebra over a field of characteristic zero, for which the supertrace of $\mu_3$ on the second argument is non-zero.
 
As a byproduct, we obtain an example of a homotopically finitely presented DG category (over a field of characteristic zero) that does not have a smooth categorical compactification, giving a negative answer to a question of To\"en. This can be interpreted as a lack of resolution of singularities in the noncommutative setup.

We also obtain an example of a proper DG category which does not admit a categorical resolution of singularities in the terminology of \cite{KL} (that is, it cannot be embedded into a smooth and proper DG category).
\end{abstract}


\maketitle

\tableofcontents

\section{Introduction}

Given a smooth algebraic variety $X$ over a field of characteristic zero, 
we have the Hodge-to-de Rham spectral sequence $E_1^{p,q}=H^q(X,\Omega^p_X)\Rightarrow H^{p+q}_{DR}(X).$ It is classically known that when $X$ is additionally proper, this spectral sequence degenerates at $E_1,$ that is, all differentials vanish. This follows from the classical Hodge theory for compact K\"ahler manifolds, and can be also proved algebraically \cite{DI}.

We recall the following fundamental result of Kaledin \cite{Ka}, see also \cite{M} for a different proof.

\begin{theo}\label{th:Kaledin_degen}\cite[Theorem 5.4]{Ka} Let $A$ be a smooth and proper DG algebra. Then the Hochschild-to-cyclic spectral sequence degenerates, so that we have an isomorphism $HP_{\bullet}(A)=HH_{\bullet}(A)((u)).$\end{theo}

Here $u$ denotes a variable of degree $2.$

When applied to $\Perf(A)\simeq \Perf(X)$ for smooth and proper variety $X,$ Theorem \ref{th:Kaledin_degen} gives exactly the classical Hodge-to-de Rham degeneration.

In this paper we study some generalizations of Hodge-to-de-Rham degeneration to DG categories which are not smooth and proper.

Recall that for a proper DG algebra $B$ one has a pairing on $HH_{\bullet}(B)\otimes HH_{\bullet}(B^{op})\to \mk,$ introduced by Shklyarov \cite{S}. Kontsevich \cite{Ko} proposed the following generalization of Theorem \ref{th:Kaledin_degen}.

\begin{conj}\label{conj:degeneration_for_proper_intro} Let $B$ be a proper DG algebra. Then the composition map \begin{equation}\label{eq:composition_for_proper}(HH_{\bullet}(B)\otimes HC_{\bullet}(B^{op}))[1]\xto{\id\otimes\delta} HH_{\bullet}(B)\otimes HH_{\bullet}(B^{op})\to\mk\end{equation} is zero.\end{conj}

Kontsevich also proposed a "dual" version of Conjecture \ref{conj:degeneration_for_proper_intro} for smooth DG algebras.

\begin{conj}\label{conj:degeneration_for_smooth_intro} Let $A$ be a smooth DG algebra. Then the composition $$K_0(A\otimes A^{op})\xto{\ch} (HH_{\bullet}(A)\otimes HH_{\bullet}(A^{op})_0\xto{\id\otimes\delta} (HH_{\bullet}(A)\otimes HC^-_{\bullet}(A^{op}))_1$$ vanishes on the class $[A]$ of the diagonal bimodule.\end{conj}

The motivation for Conjectures \ref{conj:degeneration_for_proper_intro} and  \ref{conj:degeneration_for_smooth_intro} is explained in Propositions \ref{prop:smooth_comp_implies_smooth} and \ref{prop:cat_res_implies_prop} below. Here we mention that the results of \cite{KL} imply that Conjecture \ref{conj:degeneration_for_proper_intro} holds for proper DG algebras of algebro-geometric origin: that is, for DG algebras of the form $B=\bR\End(\cF),$ where $\cF\in\Perf_Z(X)$ is a perfect complex on a separated scheme $X$ of finite type over $\mk,$ supported on a {\it proper} closed subscheme $Z\subset X.$ Similarly, the (weak version of) results of \cite{E2} imply that Conjecture \ref{conj:degeneration_for_smooth_intro} holds for smooth DG algebras of the form $\bR\End(\cG),$ where $\cG\in D^b_{coh}(X)$ is a generator of the category $D^b_{coh}(X).$

 There is a closely related question formulated by B. To\"en \cite{To1}.

\begin{ques}\label{ques:Toen} Is it true that any homotopically finitely presented DG category $\cB$ is quasi-equivalent to a quotient $\cA/\cS,$ where $\cA$ is smooth and proper, and $\cS\subset\cA$ is a full subcategory?\end{ques}

Such a quotient presentation of $\cB$ is called a smooth categorical compactification. 

In this paper we disprove both Conjectures \ref{conj:degeneration_for_proper_intro} and \ref{conj:degeneration_for_smooth_intro}. As an application, we give a negative answer to Question \ref{ques:Toen}.

The starting point for our counterexamples is to disprove the main conjecture of \cite{E}, see Section \ref{sec:disproving_very_general}. A counterexample to Conjecture \ref{conj:degeneration_for_smooth_intro} is obtained in Section \ref{sec:disproving_version_for_smooth}. It is deduced from the results of Section \ref{sec:disproving_very_general} by some trick.

Finally, a counterexample to Conjecture \ref{conj:degeneration_for_proper_intro} is obtained in Section \ref{sec:disproving_for_proper}. It is deduced from our new result on nilpotent elements in the cohomology of a DG algebra (Theorem \ref{th:nilpotence_and_factorization}), which is of independent interest. In particular, we obtain an example of a proper DG algebra $B$ such that the DG category $\Perf(B)$ cannot be fully faithfully embedded into a saturated DG category. That is, it does not have a categorical resolution of singularities in the terminology of \cite{KL}. 

Section \ref{sec:disproving_for_proper} can be read independently from Sections \ref{sec:disproving_very_general} and \ref{sec:disproving_version_for_smooth}.

{\noindent{\bf Acknowledgements.}} I am grateful to Dmitry Kaledin, Maxim Kontsevich and Bertrand To\"en for useful discussions. 

\section{Preliminaries on DG categories and $A_{\infty}$-algebras}

\subsection{DG categories} For the introduction on DG categories, we refer the reader to \cite{Ke}. The references for DG quotients are \cite{Dr, Ke2}. For the model structures on DG categories we refer to \cite{Tab1, Tab2}, and for a general introduction on model categories we refer to \cite{Ho}. Everything will be considered over some base field $\mk.$

Mostly we will consider DG categories up to a quasi-equivalence. By a functor between DG categories we sometimes mean a quasi-functor. In some cases it is convenient for us to choose a concrete DG model or a concrete DG functor. By a commutative diagram of functors we usually mean the commutative diagram in the homotopy category $\Ho(\dgcat_{\mk}).$ Finally, we denote by $\Ho_M(\dgcat_{\mk})$ the Morita homotopy category of DG categories (with inverted Morita equivalences).

All modules are assumed to be right unless otherwise stated. For a small DG category $\cC$ and a $\cC$-module $M,$ we denote by $M^{\vee}$ the $\cC^{op}$-module $\Hom_{\cC}(M,\cC).$ We denote by $M^*$ the $\cC^{op}$-module $\Hom_{\cC}(M,\mk).$

Given a small DG category $\cC,$ we denote by $D(\cC)$ its derived category of DG $\cC$-modules. This is a compactly generated triangulated category. We denote by $D_{\perf}(\cC)$ the full triangulated subcategory of perfect $\cC$-modules. It coincides with the subcategory of compact objects.

Recall from  \cite{TV} that a $\cC$-module $M$ is pseudo-perfect if for each $x\in\cC,$ the complex $M(x)$ is perfect over $\mk$ (that is, $M(x)$ has finite-dimensional total cohomology). We denote by $D_{\pspe}(\cC)\subset D(\cC)$ the full triangulated subcategory of pseudo-perfect $\cC$-modules.

For any DG category $\cC,$ we denote by $[\cC]$ its (non-graded) homotopy category, which has the same objects as $\cC,$ and the morphisms are given by $[\cC](x,y)=H^0(\cC(x,y)).$ We use the the terminology of \cite[Definition 2.4]{TV} by calling $\cC$ triangulated if the Yoneda embedding provides an equivalence $[\cC]\xto{\sim}D_{\perf}(\cC).$ In this case $[\cC]$ is a Karoubi complete triangulated category. 

We denote by $\Mod_{\cC}$ the DG category of cofibrant DG $\cC$-modules in the projective model structure (these are exactly the direct summands of semifree DG $\cC$-modules).
We have $D(\cC)\simeq [\Mod_{\cC}],$ where $D(\cC)$ is the derived category of DG $\cC$-modules. 
We denote by $\bY:\cC\hto\Mod_{\cC}$ the standard Yoneda embedding given by $\bY(x)=\cC(-,x).$



We write $\Perf(\cC)\subset \Mod_{\cC}$ (resp. $\PsPerf(\cC)\subset\Mod_{\cC}$) for the full DG subcategory of perfect (resp. pseudo-perfect) $\cC$-modules.



For a DG functor $\Phi:\cC_1\to\cC_2$ between small DG categories, we denote by $\bL\Phi^*:D(\cC_1)\to D(\cC_2),$ the derived extension of scalars functor. Its right adjoint functor (restriction of scalars) is denoted by $\Phi_*:D(\cC_2)\to D(\cC_1).$ 

We also recall from \cite[Definitions 3.6]{T} that a $\cC$-module is called quasi-representable if it is quasi-isomorphic to a representable $\cC$-module. For two DG categories $\cC,\cC',$ a $\cC\otimes\cC'$-module $M$ is called right quasi-representable if for each object $x\in\cC,$ the $\cC'$-module $M(x,-)$ is quasi-representable.

We denote by $\bR\un{\Hom}(\cC,\cC')\subset\Mod_{\cC^{op}\otimes\cC'}$ the full subcategory of right quasi-representable $\cC^{op}\otimes\cC'$-modules. By \cite[Theorem 6.1]{T}, this DG category (considered up to a quasi-equivalence) is actually the internal Hom in the homotopy category of DG categories $\Ho(\dgcat_{\mk})$ (with inverted quasi-equivalences). We have a natural quasi-functor $\Fun(\cC,\cC')\to \bR\un{Hom}(\cC,\cC'),$ where $\Fun(\cC,\cC')$ is the naive DG category of DG functors $\cC\to\cC',$ as defined in \cite{Ke}. Moreover, if $\cC$ is cofibrant, this functor is essentially surjective on the homotopy categories.

A small DG category $\cC$ is called smooth (resp. locally proper) if the diagonal $\cC\mhyphen\cC$-bimodule is perfect (resp. pseudo-perfect). Moreover, $\cC$ is called proper if it is locally proper and is Morita equivalent to a DG algebra (i.e. the triangulated category $D_{\perf}(\cC)$ has a classical generator). 



We recall the notion of a short exact sequence of DG categories.

\begin{defi}\label{defi:short_exact_dg}A pair of functors $\cA_1\xto{F_1}\cA_2\xto{F_3}\cA_3$ is said to be a (Morita) short exact sequence of DG categories if the following conditions hold

$\rm{i)}$ the composition $F_2F_1$ is homotopic to zero;

$\rm ii)$ the functor $F_1$ is quasi-fully-faithful;

$\rm iii)$ the induced quasi-functor $\bbar{F_2}:\cA_2/F_1(\cA_1)\to \cA_3$ is a Morita equivalence.
\end{defi}

In particular, a short exact sequence of DG categories induces a long exact sequence of K-groups, where $K_{\bullet}(\cA)$ is the Waldhausen K-theory \cite{W} of the Waldhausen category of cofibrant perfect $\cA$-modules. We will in fact need only the boundary map $K_1(\cA_3)\to K_0(\cA_1).$ 

\subsection{$A_{\infty}$-algebras and $A_{\infty}$-(bi)modules}

All the definitions and constructions regarding DG categories which are invariant under quasi-equivalences can be translated into the world of $A_{\infty}$-categories. For the introduction on $A_{\infty}$-categories and algebras see \cite{L-H, Ke3, KS}.

It will be sufficient for us to work with $A_{\infty}$-algebras (that is, $A_{\infty}$-categories with a single object).

In order to write down the signs in formulas it is convenient to adopt the following

{\noindent {\bf Notation.}} {\it For a collection of homogeneous elements $a_0,\dots,a_n$ of a graded vector space $A,$ and $0\leq p,q\leq n,$ we put $$l_p^q(a)=\begin{cases}|a_p|+\dots+|a_q|+q-p+1 & \text{ if }p\leq q;\\
|a_p|+\dots+|a_n|+|a_0|+\dots+|a_q|+n-p+q & \text{ if }p>q.\end{cases}.$$ If the collection starts with $a_1$ (and there is no $a_0$) we only use $l_p^q(a)$ for $1\leq p\leq q\leq n.$}

\begin{defi}A non-unital $A_{\infty}$-structure on a graded vector space $A$ is a sequence of multilinear operations $\mu_n=\mu_n^A:A^{\otimes n}\to A,$ where $\deg(\mu_n)=2-n,$ satisfying the following relations:
\begin{equation}\label{eq:A_infty_rels}\sum\limits_{i+j+k=n+1}(-1)^{l_1^i(a)}\mu_{i+k+1}(a_1,\dots,a_i,\mu_j(a_{i+1},\dots,a_{i+j}),a_{i+j+1},\dots,a_n)=0,\end{equation}
for $n\geq 0.$ Here for $1\leq p\leq q\leq n$ we put $l_p^q(a):=|a_p|+\dots+|a_q|+q-p+1.$\end{defi}

\begin{remark}In our sign convention, a non-unital DG algebra $B$ can be considered as an $A_{\infty}$-algebra, with $\mu_1(a)=-d(a),$ $\mu_2(a_1,a_2)=(-1)^{|a_1|}a_1a_2,$ and $\mu_{\geq 3}=0.$\end{remark}

\begin{defi}A non-unital $A_{\infty}$-morphism $f:A\to B$ is given by a sequence of linear maps $f_n:A^{\otimes n}\to B,$ where $\deg(f_n)=1-n,$ satisfying the following relations:
\begin{multline}\label{eq:A_infty_morphism}\sum\limits_{i_1+\dots+i_k=n}\mu_k^B(f_{i_1}(a_1,\dots,a_{i_1}),\dots,f_{i_k}(a_{i_1+\dots+i_{k-1}+1},\dots,a_n))=\\ \sum\limits_{i+j+k=n}(-1)^{l_1^i(a)}f_{i+k+1}(a_1,\dots,a_i,\mu_j^A(a_{i+1},\dots,a_{i+j}),a_{i+j+1},\dots,a_n).\end{multline}\end{defi}

Given an $A_{\infty}$-algebra $A,$ one defines the $A_{\infty}$-algebra $A^{op}$ as follows: it is equal to $A$ as a graded vector space, and we have
$$\mu_n^{A^{op}}(a_1,\dots,a_n)=(-1)^{\sigma}\mu_n^A(a_n,\dots,a_1),$$ where $\sigma=\sum\limits_{1\leq i<j\leq n}(|a_i|+1)(|a_j|+1).$

We now define the notion of an $A_{\infty}$-module.



\begin{defi}A right $A_{\infty}$-module $M$ over an $A_{\infty}$-algebra $A$ is a graded vector space with a sequence of operations $\mu_{n}^M:M\otimes A^{\otimes n-1}\to M,$ where $n>0,$ $\deg(\mu_n^M)=2-n,$ and the following relations are satisfied:
\begin{multline}\sum\limits_{i+j=n}\mu_{j+1}^M(\mu_{i+1}^M(m,a_1,\dots,a_i),a_{i+1},\dots,a_n)+\\
\sum\limits_{i+j+k=n+1}(-1)^{|m|+l_1^i(a)}\mu_{i+k+1}(m,a_1,\dots,a_i,\mu_j(a_{i+1},\dots,a_{i+j}),a_{i+j+1},\dots,a_n)=0.\end{multline}\end{defi}

We also need $A_{\infty}$-bimodules.

\begin{defi}Let $A$ and $B$ be non-unital $A_{\infty}$-algebras. An $A_{\infty}$ $A\mhyphen B$-bimodule $M$ is a graded vector space with a collection of operations $\mu_{i,j}=\mu_{i,j}^M:A^{\otimes i}\otimes M\otimes B^{\otimes j}\to M,$ where $i,j\geq 0,$ such that for any $n,m\geq 0$ and homogeneous $a_1,\dots,a_n\in A,$ $b_1,\dots,b_m\in B,$ $m\in M,$ the following relation is satisfied:
\begin{multline*}\sum\limits_{i+j+k=n+1}(-1)^{l_1^i(a)}\mu_{i+k+1,m}^M(a_1,\dots,\mu_j^A(a_{i+1},\dots,a_{i+j}),\dots,a_n,m,b_1,\dots,b_m)\\
+\sum\limits_{\substack{1\leq i\leq n+1;\\
0\leq j\leq m}}\mu_{i-1,m-j}^M(a_1,\dots,a_{i-1},\mu_{n+1-i,j}^M(a_i,\dots,a_n,m,b_1,\dots,b_j),b_{j+1},\dots,b_m)\\
+\sum\limits_{i+j+k=m+1}(-1)^{l_1^n(a)+l_1^i(b)+|m|}\mu_{n,i+k+1}^M(a_1,\dots,a_n,m,b_1,\dots,\mu_j^B(b_{i+1},\dots,b_{i+j}),\dots,b_m)=0.\end{multline*}
\end{defi}

\begin{remark}1) In our sign convention, a non-unital DG algebra $B$ can be considered as an $A_{\infty}$-algebra, with $\mu_1(b)=-d(b),$ $\mu_2(b_1,b_2)=(-1)^{|b_1|}b_1 b_2,$ and $\mu_{\geq 3}=0.$

2) If furthermore $M$ is a right DG $B$-module, then the $A_{\infty}$ $B$-module structure on $M$ is given by $\mu_1^M(m)=d(m),$ $\mu_2^M(m,a)=(-1)^{|m|+1}ma,$ and $\mu_{\geq 3}^M=0.$

3) If $A$ is another non-unital DG algebra, and $M$ is a DG $A\mhyphen B$-bimodule, then the $A_{\infty}$ $A\mhyphen B$-bimodule structure on $M$ is given by $\mu_{0,0}^M(m)=d(m),$ $\mu_{1,0}^M(a,m)=am,$ $\mu_{0,1}^M(m,b)=(-1)^{|m|+1}mb,$ and $\mu_{i,j}^M=0$ for $i+j\geq 2.$\end{remark}

We now recall the strict unitality.

\begin{defi}1) A non-unital $A_{\infty}$-algebra $A$ is called strictly unital if there is a (unique) element $1=1_A\in A$ such that $\mu_1(1)=0,$ $\mu_2(1,a)=a=(-1)^{|a|}\mu_2(a,1)$ for any homogeneous element $a\in A,$ and for $n\geq 3$ we have $\mu_n(a_1,\dots,a_n)=0$ if at least one of the arguments $a_i$ equals $1.$ 

2) A non-unital $A_{\infty}$-morphism $f:A\to B$ between strictly unital $A_{\infty}$-algebras is called strictly unital if $f_1(1_A)=1_B,$ and for $n\geq 2$ we have $f_n(a_1,\dots,a_n)=0$ if at least one of the arguments $a_i$ equals $1.$

3) Given a strictly unital $A_{\infty}$-algebra $A,$ an $A_{\infty}$ $A$-module $M$ is called strictly unital if $\mu_2^M(m,1)=(-1)^{|m|+1}m,$ and for $n\geq 3$ we have $\mu_n^M(m,a_1,\dots,a_{n-1})=0$ if at least one of $a_i$'s equals $1.$ 

4) Given strictly unital $A_{\infty}$-algebras $A,$ $B,$ an $A_{\infty}$ $A\mhyphen B$-bimodule is called strictly unital if $\mu_{1,0}^M(1_A,m)=m,$ $\mu_{0,1}^M(m,1_B)=(-1)^{|m|+1}m,$ and for $k+l\geq 2$ we have $\mu_{k,l}(a_1,\dots,a_k,m,b_1,\dots,b_l)=0$ if at least one of $a_i$'s equals $1_A$ or at least one of $b_j$'s equals $1_B.$
\end{defi}

From now on, all $A_{\infty}$-algebras and (bi)modules will be strictly unital

Given a strictly unital $A_{\infty}$-algebra $A,$ we define the DG category $\Mod^{\infty}\mhyphen A$  whose objects are $A_{\infty}$-modules and the morphisms are defined as follows. Given $M,N\in \Mod^{\infty}\mhyphen A,$ we put $$\Hom^{\infty}_A(M,N)^{gr}:=\prod\limits_{n\geq 0}\Hom_{\mk}(M\otimes A[1]^{\otimes n},N),$$ and the differential is given by
\begin{multline*}d(\varphi)_n(m,a_1,\dots,a_n)=\sum\limits_{i=0}^n \mu_{n-i+1}^N(\varphi_i(m,a_1,\dots,a_i),a_{i+1},\dots,a_n)\\
-\sum\limits_{i=0}^n(-1)^{|\varphi|}\varphi_{n-i+1}(\mu_{i+1}^M(m,a_1,\dots,a_i),a_{i+1},\dots,a_n)\\
-\sum\limits_{1\leq i\leq j\leq n}(-1)^{|\varphi|+|m|+l_1^{i-1}(a)}\varphi_{n+i-j-1}(m,a_1,\dots,\mu_{j-i+1}^A(a_i,\dots,a_j),\dots,a_n).\end{multline*} 
The composition is given by
$$(\varphi\psi)_n(m,a_1,\dots,a_n)=\sum\limits_{i=0}^n\varphi_{n-i}(\psi_{i}(m,a_1,\dots,a_i),a_{i+1},\dots,a_n).$$

Given a unital DG algebra $B,$ we denote by $\PsPerf(B)\subset \Mod^{\infty}\mhyphen B$ the full DG subcategory formed by pseudo-perfect DG modules. We have $[\PsPerf(B)]\simeq D_{\perf}(B).$

\begin{remark}\label{rem:what_is_bimodule}Let $A,B$ be $A_{\infty}$-algebras. 

1) An $A_{\infty}$ $A\mhyphen B$-bimodule structure on a graded vector space $M$ is equivalent to the following data:
\begin{itemize}
\item the right $A_{\infty}$ $B$-module structure on $M;$
\item the $A_{\infty}$-morphism $f:A\to \End^{\infty}_B(M).$
\end{itemize}
Namely, given an $A_{\infty}$-bimodule $M,$ the induced $B$-module structure is given by $\mu_n^M=\mu_{0,n-1}^M,$ and the $A_{\infty}$-morphism is given by $f_n(a_1,\dots,a_n)(m,b_1,\dots,b_l)=\mu_{n,l}(a_1,\dots,a_n,m,b_1,\dots,b_l).$ 

2) Also, an $A_{\infty}$-bimodule structure is equivalent to an 
\end{remark}

We finally define a technically useful notion of an an $A_{\infty}$-bimorphism of (strictly unital) $A_{\infty}$-algebras $f:(A,B)\to C.$ It is given by the linear maps $f_{r,s}:A^{\otimes r}\otimes B^{\otimes s}\to C,$ where $r,s\geq 0,$ $r+s>0,$ so that the following relations are satisfied:
\begin{multline}\sum\limits_{\substack{0=r_0\leq r_1\leq\dots\leq r_k=r;\\
0=s_0\leq s_1\leq\dots\leq s_k=s}}(-1)^{\sigma}\mu_k^C(f_{r_1,s_1}(a_1,\dots,a_{r_1};b_1,\dots,b_{s_1}),\dots,\\
f_{r-r_{k-1},s-s_{k-1}}(a_{r_{k-1}+1}\dots,a_{r};b_{s_{k-1}+1},\dots,b_s))=\\
\sum\limits_{i+j+k=r}(-1)^{l_1^i(a)}f_{i+k+1,s}(a_1,\dots,a_i,\mu_j(a_{i+1},\dots,a_{i+j}),a_{i+j+1},\dots,a_r;b_1,\dots,b_s)+\\
\sum\limits_{i+j+k=s}(-1)^{l_1^r(a)+l_1^i(b)}f_{r,i+k+1}(a_1,\dots,a_r;b_1,\dots,b_i,\mu_j(b_{i+1},\dots,b_{i+j}),b_{i+j+1},\dots,b_s),\end{multline}
where $\sigma=\sum\limits_{1\leq p<q\leq k}l_{r_{q-1}+1}^{r_q}(a)l_{s_{p-1}+1}^{s_p}(b).$ We require that $f_{1,0}(1_A)=1_C=f_{0,1}(1_B),$ and for $k+l\geq 2$ $f_{k,l}(a_1,\dots,a_k,b_1,\dots,b_l)=0$ if at least one of $a_i$'s equals $1_A,$ or at least one of $b_j$'s equals $1_B.$ 

\begin{remark}One can similarly define $A_{\infty}$ $n$-morphisms $(A_1,\dots,A_n)\to B,$ so that the category of $A_{\infty}$-algebras becomes a (non-symmetric) pseudo-monoidal category. In particular, the $A_{\infty}$-morphisms can be composed with $A_{\infty}$-morphisms in the natural way.\end{remark}

\begin{remark}\label{rem:bimodule_as_bimorphism} If a graded vector space $M$ is given a differential $d,$ then an $A_{\infty}$-bimodule structure on $M$ (with $\mu_1^M=d$) is equivalent to an $A_{\infty}$-bimorphism $f:(A,B^{op})\to\End_{\mk}(M).$ Given such an $A_{\infty}$-bimorphism, one puts $$\mu_{r,s}^M(a_1,\dots,a_r,m,b_1,\dots,b_s):=(-1)^{l}f_{r,s}(a_1,\dots,a_r,b_s,\dots,b_1)(m),$$
where $l=l_1^s(b)\cdot|m|+\sum\limits_{1\leq p<q\leq s}(|b_p|+1)(|b_q|+1).$\end{remark}

The diagonal $A_{\infty}$ $A\mhyphen A$-bimodule is given by $A$ as a graded vector space, and we have $$\mu_{i,j}(a_1,\dots,a_i,b,c_1,\dots,c_j)=(-1)^{l_1^i(a)+1}\mu_{i+j+1}^A(a_1,\dots,a_i,b,c_1,\dots,c_j).$$

Finally, we mention the gluing of $A_{\infty}$-algebras. Let $M$ be an $A_{\infty}$ $A\mhyphen B$-bimodule. We denote by 
$\begin{pmatrix}
B & 0\\
M & A
\end{pmatrix}$ the $A_{\infty}$-algebra $C$ which equals $A\oplus B\oplus M$ as a graded vector space, so that the non-zero components of $\mu_n^C$ are given by $\mu_n^A,$ $\mu_n^B,$ and
$$(-1)^{l_1^i(a)+1}\mu_{i,j}(a_1,\dots,a_i,m,b_1,\dots,b_j),\quad i+j+1=n,$$
where $a_1,\dots,a_i\in A,$ $b_1,\dots,b_j\in B.$ 

\section{Preliminaries on the Hochschild complex, pairings and copairings}

In this section all $A_{\infty}$-algebras are strictly unital. For an $A_{\infty}$-algebra $A,$ we put $\bbar{A}:=A/\mk\cdot 1_A.$

The mixed Hochschild complex (see \cite{Ke2, KS}) $(C_{\bullet}(A),b,B)$ of an $A_{\infty}$-algebra $A$ is given as a graded vector space by $$C_{\bullet}(A):=\bigoplus\limits_{n\geq 0}A\otimes (\bbar{A}[1])^{\otimes n}.$$ For convenience we write $(a_0,\dots,a_n)$ instead of $a_0\otimes\dots\otimes a_n\in C_{\bullet}(A).$

The Hochschild differential is given by \begin{multline}b(a_0,\dots,a_n)=\sum\limits_{0\leq i\leq j\leq n}(-1)^{l_0^{i-1}(a)+1}(a_0,\dots,\mu_{j-i+1}(a_i,\dots,a_j),\dots,a_n)+\\
\sum\limits_{0\leq p<q\leq n}(-1)^{l_0^{q-1}(a)l_q^n(a)+1}(\mu_{n+p+2-q}(a_q,\dots,a_n,a_0,\dots,a_p),a_{p+1},\dots,a_{q-1}).\end{multline}
The Connes-Tsygan  differential $B$ (see \cite{Co, FT, Ts}) is given by
$$B(a_0,a_1,\dots,a_n)=\sum\limits_{0\leq i\leq n}(-1)^{l_0^{i-1}(a)l_i^n(a)+1}(1,a_i,\dots,a_n,a_0,\dots,a_{i-1}).$$

The Hochschild complex can be more generally defined for $A_{\infty}$-categories, and is Morita invariant \cite{KS}. We refer to \cite{KS} for the definition of cyclic homology $HC_{\bullet},$ negative cyclic homology $HC^{-}_{\bullet}$ and $HP_{\bullet}.$ In this paper we will in fact deal only with the first differential of the Hochschild-to-cyclic spectral sequence, which is the map $B:HH_n(A)\to HH_{n+1}(A)$ induced by the Connes-Tsygan differential. 

We recall the natural pairings and co-pairings on $HH_{\bullet}(A).$ Let us restrict ourselves to DG algebras for a moment. Given a DG algebra $A,$ we have a Chern character $\ch:K_n(A)\to HH_n(A)$ (see \cite{CT}; the Chern character naturally lifts to $HC^{-}(A)$), but we will not need this).

In particular, given DG algebras $A,B$ and an object $M\in D_{\perf}(A\otimes B),$ we have a copairing $$\ch(M)\in (HH_{\bullet}(A)\otimes HH_{\bullet}(B))_0\cong HH_0(A\otimes B).$$ This copairing is used in the formulation of Conjecture \ref{conj:degeneration_for_smooth_intro} for $A=B^{op}$ being smooth, and $M=A.$

Dually \cite{S}, if we have DG algebras $A$ and $B,$ and an object $M\in D_{\pspe}(A^{op}\otimes B^{op}),$ then we have a pairing (of degree zero) $$HH_{\bullet}(A)\otimes HH_{\bullet}(B)\to HH_{\bullet}(A\otimes B)\to HH_{\bullet}(\End_{\mk}(M))\to \mk$$
(the last map is an isomorphism if and only if $M$ is not acyclic). In the formulation of Conjecture \ref{conj:degeneration_for_proper_intro} this pairing is used for $A=B^{op}$ proper, and $M=A.$ In this case we denote the pairing by $\langle \cdot,\cdot\rangle.$

We would like to obtain an explicit formula for the pairing in the $A_{\infty}$-setting. The reader who is not interested in (or is already familiar with) the details can skip to Corollary \ref{cor:corollary_on_m_3} which is essentially all we need.

Let $A,B,C$ be $A_{\infty}$-algebras. Suppose that we are given an$A_{\infty}$-bimorphism $f:(A,B)\to C.$ We would like to define an explicit map of complexes $$f_*:C_{\bullet}(A)\otimes C_{\bullet}(B)\to C_{\bullet}(C).$$ It is given by
\begin{multline}\label{eq:map_on HH_induced by_bimorphism}f_*((a_0,\dots,a_n)\otimes (b_0,\dots,b_m))=\\
\sum\limits_{\substack{0\leq i_0\leq \dots\leq i_k\leq n;\\0\leq j_0\leq\dots\leq j_k\leq m;\\
0\leq p< q\leq k}}(-1)^{\veps(i_0,\dots,i_k,j_1,\dots,j_k,p,q)}(\mu_{k+p+2-q}(f_{i_{q+1}-i_q,j_q-j_{q-1}}(a_{i_q+1},\dots,a_{i_{q+1}},b_{j_{q-1}+1},\dots,b_{j_q}),\\ \dots,
 f_{i_{p+1}-i_p,j_p-j_{p-1}}(a_{i_p+1},\dots,a_{i_{p+1}},b_{j_{p-1}+1},\dots,b_{j_p})),\\
 f_{i_{p+2}-i_{p+1},j_{p+1}-j_p}(a_{i_{p+1}+1},\dots,a_{i_{p+2}},b_{j_p+1},\dots,b_{j_{p+1}}),\dots,\\
 f_{i_q-i_{q-1},j_{q-1}-j_{q-2}}(a_{i_{q-1}+1},\dots,a_{i_q},b_{j_{q-2}+1},\dots,b_{j_{q-1}})),\end{multline}
where \begin{multline*}\veps(i_0,\dots,i_k,j_1,\dots,j_k,p,q)=l_0^m(a)+l_{i_q+1}^{n}(a)l_{0}^{i_q}(a)+l_{j_{q-1}+1}^{m}(b)l_{0}^{j_{q-1}}(b)+1+\\
\sum\limits_{s=1}^k l_{i_{q+s}+1}^{i_{q+s+1}}(a)l_{j_{q-1}+1}^{j_{q+s-1}}(b).\end{multline*}
 
In this summation we mean that $i_{s+k+1}=i_s,$ $j_{s+k+1}=j_s,$ $a_{s+n+1}=a_s,$ $b_{s+m+1}=b_s.$ Also, we require that for all $s=1,\dots,k-1$ we have $(i_{s+1}-i_s)+(j_s-j_{s-1})>0,$ so that we don't get the (non-existing) $f_{0,0}$ anywhere.

\begin{remark}Suppose that we are in the special situation when $A,$ $B$ and $C$ are DG algebras, and the $A_{\infty}$-bimorphism $f$ has only two non-zero components $f_{1,0}$ and $f_{0,1}.$ This is equivalent to a DG algebra morphism $A\otimes B\to C,$ which we still denote by $f.$

The map given by \eqref{eq:map_on HH_induced by_bimorphism} is obtained by composing the map $C_{\bullet}(A\otimes B)\to C_{\bullet}(C)$ with the Eilenberg-Zilber map $EZ:C_{\bullet}(A)\otimes C_{\bullet}(B)\to C_{\bullet}(A\otimes B).$
\end{remark}

\begin{prop}\label{prop:explicit_pairing_via_str}Let $A_1$ and $A_2$ be strictly unital $A_{\infty}$-algebras, and $M$ a finite dimensional strictly unital $A_{\infty}$ $A_1\mhyphen A_2$-bimodule (we require that $\dim\oplus_{n}\dim(M^n)<\infty$). Then the composition map $$\psi:HH_{\bullet}(A_1)\otimes HH_{\bullet}(A_2^{op})\xto{\id\otimes B}HH_{\bullet}(A_1)\otimes HH_{\bullet}(A_2^{op})\to HH_{\bullet}(\End(V))\to \mk$$ is given by the following explicit formula:
\begin{multline*}\psi((a_0,\dots,a_n)\otimes (b_0,\dots,b_m))=\str_M(m\mapsto\\
\mapsto (-1)^{l_0^m(b)\cdot|m|}\sum\limits_{\substack{0\leq i\leq n;\\ 0\leq j\leq m}}(-1)^{\sigma_{i,j}}\mu_{n+1,m+1}(a_i,\dots,a_k,\dots,a_{i-1},m,b_j,\dots,b_0,b_l,\dots,b_{j+1}),\end{multline*}
where $$\sigma_{i,j}=l_0^n(a)+l_0^{i-1}(a)l_i^n(a)+\sum\limits_{0\leq p<q\leq j}(|b_p|+1)(|b_q|+1)+\sum\limits_{j+1\leq p<q\leq m}(|b_p|+1)(|b_q|+1).$$\end{prop}

\begin{proof}Recall that for a finite-dimensional complex $V$ the natural map $HH_{\bullet}(\End_{\mk}(V))\to\mk$ (which is an isomorphism if and only if $M$ is not acyclic) is given by the following morphism of complexes $C_{\bullet}(\End_{\mk}(V))\to \mk:$
$$(a_0,\dots,a_k)\mapsto \begin{cases}\str_M(a_0) & \text{ for }k=0,\,|a_0|=0;\\
0 & \text{otherwise.}\end{cases}.$$

The result follows by applying the formula \eqref{eq:map_on HH_induced by_bimorphism} and Remark \ref{rem:bimodule_as_bimorphism} (and taking the strict unitality into account).\end{proof}

Finally, we mention one particular corollary which we need in this paper.

\begin{cor}\label{cor:corollary_on_m_3}Let $A$ be a finite-dimensional non-unital $A_{\infty}$-algebra, and $a,b\in A$ are closed homogeneous elements such that $|a|+|b|=1.$ If we consider $a$ and $b$ as classes in $HH_{\bullet}(A)$ and $HH_{\bullet}(A^{op})$ respectively. Then $$\langle a,B(b)\rangle=(-1)^{|a|+1}\str_A(v\mapsto (-1)^{(|b|+1)\cdot|v|}\mu_3(a,v,b)).$$\end{cor}

\begin{proof}This follows immediately from Proposition \ref{prop:explicit_pairing_via_str}.\end{proof}

\section{A counterexample to the generalized degeneration conjecture}
\label{sec:disproving_very_general}

We recall the main conjecture of \cite{E}.

\begin{conj}\label{conj:very_general}\cite[Conjecture 1.3 for $n=0$]{E} Let $\cB$ and $\cC$ be small DG categories over a field $\mk$ of characteristic zero. Then the composition map \begin{equation}\label{eq:map_phi_0}\varphi_0:K_0(\cB\otimes\cC)\xto{\ch} (HH_{\bullet}(\cB)\otimes HH_{\bullet}(\cC))_0\xto{\id\otimes\delta} (HH_{\bullet}(\cB)\otimes HC^-_{\bullet}(\cC))_1\end{equation} is zero.\end{conj}

In this section we construct a counterexample to Conjecture \ref{conj:very_general}. We put $\Lambda_1=\mk\langle\xi\rangle/\xi^2,$ where $|\xi|=1,$ and (automatically) $d\xi=0.$ We have a quasi-equivalence $\Perf(\Lambda_1)\simeq\Perf_{\{0\}}(\A_{\mk}^1)$ (the free $\Lambda_1$-module of rank $1$ corresponds to the skyscraper sheaf $\cO_{0}$). In particular, we have a short exact sequence
\begin{equation}\label{eq:ex_sec_on_A^1}0\to\Perf(\Lambda_1)\to \Perf(\A^1)\to\Perf(\bG_m)\to 0\end{equation}

 We also denote by $\mk[\veps]:=\mk[t]/t^2$ the algebra of dual numbers ($|\veps|=0,$ $d\veps=0$). Let us denote by $x$ the coordinate on $\A^1,$ and put $T:=\Spec(\mk[\veps]).$ Tensoring \eqref{eq:ex_sec_on_A^1} by $\mk[\epsilon]$ (and taking perfect complexes), we obtain another short exact sequence:
\begin{equation}\label{eq:ex_sec_times_T}0\to \Perf(\Lambda_1\otimes \mk[\epsilon])\to \Perf(\A^1\times T)\to \Perf(\bG_m\times T)\to 0.\end{equation}

Now let us take the Cartier divisor $D:=\{x+\veps=0\}\subset \A^1\times T.$ This is well-defined since $x+\veps$ is not a zero divisor in $\mk[x]\otimes\mk[\veps].$ Moreover, we have $D\cap (\bG_m\times T)=\emptyset,$ since $x+\veps$ is invertible in $\mk[x^{\pm 1}]\otimes\mk[\veps]:$ we have $(x+\veps)(x^{-1}-x^{-2}\veps)=1.$ Therefore, by \eqref{eq:ex_sec_times_T}, we may and will consider $\cO_D$ as an object of $\Perf(\Lambda_1\otimes\mk[\veps]).$

\begin{theo}\label{th:disproving_very_general}Conjecture \ref{conj:very_general} does not hold for the DG algebras $\Lambda_1$ and $\mk[\veps].$ Namely, we have $\varphi_0([\cO_D])\ne 0,$ where $\varphi_0$ is defined in \eqref{eq:map_phi_0}.\end{theo}

\begin{proof}We will prove a stronger statement: $\psi_0([\cO_{D}])\ne 0,$ where $\psi_0$ is the composition $$K_0(\Lambda_1\otimes\mk[\veps])\xto{\ch}(HH_{\bullet}(\Lambda_1)\otimes HH_{\bullet}(\mk[\veps]))_0\xto{\id\otimes B} (HH_{\bullet}(\Lambda_1)\otimes HH_{\bullet}(\mk[\veps]))_1.$$

We use the notation $d_{dR}$ for the de Rham differential in order to avoid confusion with differentials in DG algebras.

First let us identify the Hochschild homology of $\Lambda_1.$ Applying the long exact sequence in Hochschild homology to \eqref{eq:ex_sec_on_A^1}, we see that $$HH_{-1}(\Lambda_1)=\mk[x^{\pm 1}]/\mk[x],\text{ and }HH_0(\Lambda_1)=\mk[x^{\pm 1}]d_{dR}x/\mk[x]d_{dR}x,$$ and $HH_i(\Lambda_1)=0$ for $i\not\in\{-1,0\}.$

Further, for any commutative $\mk$-algebra $R$ we have $HH_0(R)=R,$ and $HH_1(R)=\Omega^1_{R/\mk},$ (and the Connes differential $B:HH_0(R)\to HH_1(R)$ is given by the de Rham differential). In particular, we have $HH_0(\mk[\veps])=\mk[\veps],$ and $HH_1(\mk[\veps])=\mk\cdot d_{dR}\veps$ (and we do not need $HH_{\geq 2}(\mk[\veps])$ for our considerations).

{\noindent{\bf Claim.}} {\it Within the above notation, we have $\ch(\cO_D)=\frac{d_{dR}x}{x}\otimes 1-\frac{d_{dR}x}{x^2}\otimes \veps+\frac{1}{x}\otimes d_{dR}\veps.$}

\begin{proof}As we already mentioned, the function $x+\veps$ is invertible on $\bG_m\times T,$ hence it gives an element $\alpha\in K_1(\bG_m\times T).$ Moreover, the boundary map $$K_1(\bG_m\times T)\to K_0(\Lambda_1\otimes\mk[\veps])$$ sends $\alpha$ to $[\cO_D].$ We have $\ch(\alpha)=d_{dR}\log(x+\veps)\in\Omega^1_{\bG_m\times T}=HH_1(\bG_m\times T).$ Explicitly, we have
$$d_{dR}\log(x+\veps)=(x^{-1}-x^{-2}\veps)d_{dR}(x+\veps)=\frac{d_{dR}x}{x}-\frac{\veps d_{dR}x}{x^2}+\frac{d_{dR}\veps}{x}.$$
Applying the boundary map $HH_1(\bG_m\times T)\to HH_0(\Lambda_1\otimes \mk[\veps]),$ we obtain the desired formula for $\ch(\cO_D).$\end{proof}

It follows from Claim that $$(\id\otimes B)(\ch([\cO_D]))=-\frac{d_{dR}x}{x^2}\otimes d_{dR}\veps\ne 0.$$ This proves the theorem.
\end{proof}

\section{A counterexample to Conjecture \ref{conj:degeneration_for_smooth_intro}}
\label{sec:disproving_version_for_smooth}

In this section we disprove Conjecture \ref{conj:degeneration_for_smooth_intro}.

\begin{prop}\label{prop:smooth_comp_implies_smooth}Let $B$ be a smooth DG algebra and $F:\Perf(A)\to \Perf(B)$ a localization functor, where $A$ is a smooth and proper DG algebra. Then Conjecture \ref{conj:degeneration_for_proper_intro} holds for $B.$\end{prop}

\begin{proof}This is actually explained in \cite[proof of Theorem 4.6]{E}. We explain the argument for completeness. The localization assumption implies that $(F\otimes F^{op})^*(I_A)=I_B.$ In particular, the map $HH_{\bullet}(A)\otimes HH_{\bullet}(A^{op})\to HH_{\bullet}(B)\otimes HH_{\bullet}(B^{op})$ takes $\ch(I_A)$ to $\ch(I_B).$ It remains to apply the commutative diagram
$$\begin{CD}HH_{\bullet}(A)\otimes HH_{\bullet}(A^{op})@>\id\otimes\delta >> HH_{\bullet}(A)\otimes HC^{-}_{\bullet}(A^{op})[-1]\\
@VVV @VVV\\
HH_{\bullet}(B)\otimes HH_{\bullet}(B^{op})@>\id\otimes\delta >> HH_{\bullet}(B)\otimes HC^{-}_{\bullet}(B^{op})[1],\end{CD}$$
and Theorem \ref{th:Kaledin_degen} applied to $A.$
\end{proof}

We have the following corollary, mentioned in the introduction.

\begin{cor}\label{cor:conj_for_smooth_holds_for_D^b_coh} Let $X$ be a separated scheme of finite type over $\mk,$ and $\cG\in D^b_{coh}(X)$ -- a generator. Then Conjecture \ref{conj:degeneration_for_smooth_intro} holds for the smooth DG algebra $A=\bR\End(\cG).$\end{cor}

\begin{proof}Indeed, by \cite[Theorem 1.8 1)]{E2}, there is a localization functor of the form $D^b_{coh}(Y)\to D^b_{coh}(X),$ where $Y$ is a smooth projective algebraic variety over $\mk.$ The result follows by Proposition \ref{prop:smooth_comp_implies_smooth}. Note that here we don't even need to apply Theorem \ref{th:Kaledin_degen} since we only use the classical Hodge-to-de Rham degeneration for $Y.$\end{proof}

\begin{remark}In fact, in the formulation of Proposition \ref{prop:smooth_comp_implies_smooth} we could weaken the assumption on the functor $F$ to be a localization, requiring it only to be a homological epimorphism, which means that the functor $D(A)\to D(B)$ a localization, see \cite[Section 3]{E2}. Then in the proof of Corollary \ref{cor:conj_for_smooth_holds_for_D^b_coh} we can apply the corresponding weakened version of \cite[Theorem 1.8 1)]{E2} which is much easier to prove.\end{remark}

Clearly, Conjecture \ref{conj:degeneration_for_smooth_intro} is a special case of Conjecture \ref{conj:very_general}. On the other hand, it was proved in \cite{E} that Conjectures \ref{conj:very_general} and \ref{conj:degeneration_for_smooth_intro} are actually equivalent (more precisely, this follows from the proof of \cite[Theorem 4.6]{E}). However, deducing an explicit counterexample to Conjecture \ref{conj:degeneration_for_smooth_intro} along the lines of \cite{E}  would require some computations, which we wish to avoid. Instead, we use some trick.

Let us take some elliptic curve $E$ over $\mk,$ with a $\mk$-rational point $p\in E(\mk).$ Choosing a local parameter $x\in\cO_{E,p},$ we get an identification $\Perf(\Lambda_1)\simeq\Perf_{\{p\}}(E)\subset \Perf(E).$ Let us choose some generator $\cF\in\Perf(E)$ (e.g. $\cF=\cO_E\oplus\cO_p$), and put $B_E=\bR\End(\cF),$ so that $\Perf(B_E)\simeq\Perf(E).$ We denote by $F:\Perf(\Lambda_1)\hto \Perf(B_E)$ the resulting embedding.

Further, we denote by $C$ the semi-free DG algebra $\mk\langle t_1,t_2\rangle,$ with $|t_1|=0,$ $|t_2|=-1,$ $dt_1=0,$ and $dt_2=t_1^2.$

We take the object $M\in\Perf(\Lambda_1\otimes C\otimes C)$ whose image in $\Perf(\mk[x]\otimes C\otimes C)$ is given by $$Cone(\mk[x]\otimes C^{\otimes 2}\xto{x\otimes 1^{\otimes 2}+1\otimes t_1^{\otimes 2}}\mk[x]\otimes C^{\otimes 2}).$$ As in the previous section, we see that $M$ is well-defined since the element $$x\otimes 1^{\otimes 2}+1\otimes t_1^{\otimes 2}\in H^0(\mk[x^{\pm 1}]\otimes C\otimes C)=\mk[x^{\pm 1}]\otimes\mk[\veps]\otimes\mk[\veps]$$ is invertible.

Finally, we put $N:=(F\otimes\id_C^{\otimes 2})^*(M)\in \Perf(B_E\otimes C\otimes C).$

\begin{theo}\label{th:disproving_for_smooth} 1) Within the above notation, the dg algebra $$A:=\begin{pmatrix}
B_E\otimes C & 0\\
N & C^{op}
\end{pmatrix}$$ is homotopically finitely presented (hence smooth), but it does not satisfy Conjecture \ref{conj:degeneration_for_smooth_intro}.

2) The DG category $\Perf(A)$ gives a negative answer to Question \ref{ques:Toen}.\end{theo}





\begin{proof}First, by Proposition \ref{prop:smooth_comp_implies_smooth} we see that 2) reduces to 1).

We now prove 1). The homotopy finiteness of $A$ follows from \cite[Proposition 5.15]{E2} (gluing of homotopically finite DG algebras by a perfect bimodule is again homotopically finite). 

The functor $F:\Perf(\Lambda_1)\to \Perf(B_E)\simeq\Perf(E)$ induces a map $HH_F$ in Hochschild homology. We need the following values of $HH_F.$ First, the morphism $HH_F:HH_0(\Lambda_1)\to HH_0(E)=H^0(\cO_E)\oplus H^1(\omega_E)\cong\mk\oplus\mk$ is given by $$\frac{d_{dR}x}{x^n}\mapsto \begin{cases}(0,1) & \text{ for }n=1;\\
0 & \text{ for }n>1.\end{cases}.$$
Further, the morphism $HH_F:HH_{-1}(\Lambda_1)\to HH_{-1}(E)=H^1(\cO_E)$ does not vanish on $x^{-1}$ (because there is no rational function on $E$ having single simple pole at $p$). We denote the image $HH_F(x^{-1})$ by $[x^{-1}].$

To prove 1), it suffices to show that $(\id\otimes\id\otimes B)(\ch(N))\in (HH_{\bullet}(\Lambda_1)\otimes HH_{\bullet}(C)^{\otimes 2})_{1}$ is non-zero. We have a natural projection $\pi:C\to H^0(C)\cong \mk[\veps].$ Let us put $\bar{N}:=(\id\otimes\pi^*\otimes\pi^*)(N)\in\Perf(E\times T\times T).$ Then $\bar{N}$ is naturally isomorphic to $\cO_{D'},$ where $D'\subset E\times T\times T$ is a Cartier divisor, set-theoretically contained in $\{p\}\times T\times T,$ and given locally by the equation $x\otimes 1^{\otimes 2}+1\otimes \veps^{\otimes 2}=0.$ The computation from Section \ref{sec:disproving_very_general} implies that $$\ch(\bar{N})=(0,1)\otimes 1^{\otimes 2}+[x^{-1}]\otimes d_{dR}\veps\otimes\veps+[x^{-1}]\otimes \veps\otimes d_{dR}\veps.$$ Therefore, we obtain $$(\id\otimes \id\otimes B)(\ch(\bar{N}))=[x^{-1}]\otimes d_{dR}\veps\otimes d_{dR}\veps\ne 0.$$ By functoriality, this implies $(\id\otimes \id\otimes B)(\ch(N))\ne 0.$ This proves 1).
\end{proof}

\section{A counterexample to Conjecture \ref{conj:degeneration_for_proper_intro}}
\label{sec:disproving_for_proper}

In this section we disprove Conjecture \ref{conj:degeneration_for_proper_intro}.







More precisely, we will construct an example of a minimal finite-dimensional $A_{\infty}$-algebra $B$ and two elements $a,b\in B,$ such that $|a|+|b|=1,$ and
$$\str_B(v\mapsto (-1^{(|b|+1)|v|})\mu_3(a,v,b))\ne 0,$$ thus disproving Conjecture \ref{conj:degeneration_for_proper_intro} (by Corollary \ref{cor:corollary_on_m_3}). 

We first mention the following observation, which in fact motivates Conjecture \ref{conj:degeneration_for_proper_intro}.

\begin{prop}\label{prop:cat_res_implies_prop}Let $B$ be a proper DG algebra and $\Perf(B)\hto \Perf(A)$ a quasi-fully-faithful functor, where $A$ is a smooth and proper DG algebra. Then Conjecture \ref{conj:degeneration_for_proper_intro} holds for $B.$\end{prop}

\begin{proof}Indeed this follows from the commutative diagram
$$\begin{CD}HH_{\bullet}(B)\otimes HC_{\bullet}(B^{op})[1] @>{\id\otimes\delta}>> HH_{\bullet}(B)\otimes HH_{\bullet}(B^{op})@>>> \mk;\\
@VVV @VVV @V\id VV\\
HH_{\bullet}(A)\otimes HC_{\bullet}(A^{op})[1] @>{\id\otimes\delta}>> HH_{\bullet}(A)\otimes HH_{\bullet}(A^{op})@>>> \mk\end{CD}$$
and Theorem \ref{th:Kaledin_degen} applied to $A.$
\end{proof}

We have the following corollary, mentioned in the introduction.

\begin{cor}\label{cor:conj_for_smooth_holds_for_D^b_coh} Let $X$ be a separated scheme of finite type over $\mk,$ and $Z\subset X$ a closed proper subscheme. For any object $\cF\in \Perf_Z(X),$ Conjecture \ref{conj:degeneration_for_proper_intro} holds for the proper DG algebra $B=\bR\End(\cF).$\end{cor}

\begin{proof}Choosing some compactification $X\subset \bar{X}$ (which exist by Nagata's compactification theorem \cite{N}), we get $\Perf_Z(X)\simeq \Perf(\bar{X}).$ Thus, we may and will assume $X=\bar{X}=Z.$ Then the result follows by applying Proposition \ref{prop:cat_res_implies_prop} with \cite{KL}[Theorem 6.12]. As in the proof of Corollary \ref{cor:conj_for_smooth_holds_for_D^b_coh}, we only use here the classical Hodge-to-de Rham degeneration.\end{proof}

The crucial point is the following theorem, which is of independent interest.

\begin{theo}\label{th:nilpotence_and_factorization} 1) Let $A$ be a DG algebra, and $a\in H^0(A)$ a nilpotent element. Then the corresponding morphism $f:\mk[x]\to A$ (where $|x|=0$) in $\Ho(\dgalg_{\mk})$ factors through $\mk[x]/x^n$ for a sufficiently large $n.$

2) If moreover $a^2=0$ in $H^0(A),$ then it suffices to take $n=6.$\end{theo}

Before we prove Theorem \ref{th:nilpotence_and_factorization}, we show how it allows to construct a counterexample to Conjecture \ref{conj:degeneration_for_proper_intro}.

\begin{theo}1) Let us denote by $y$ the variable of degree $1.$ Then there exists an $A_{\infty}$ $\mk[y]/y^3\mhyphen \mk[x]/x^6$-bimodule structure on the $1$-dimensional vector space $V=\mk\cdot z$ (where $|z|=0$) such that $\mu_3^V(x,z,y)=z.$ In particular, in the glued $A_{\infty}$-algebra
$$B=\begin{pmatrix}
\mk[y]/y^3 & 0\\
V & \mk[x]/x^6
\end{pmatrix}$$
we have $\str(v\mapsto \mu_3(x,v,y))=1.$ Therefore, by Corollary \ref{cor:corollary_on_m_3} this $A_{\infty}$-algebra (and any quasi-isomorphic DG algebra) does not satisfy Conjecture \ref{conj:degeneration_for_proper_intro}.

2) In particular, the proper DG category $\Perf^{\infty}(B)$ does not have a categorical resolution of singularities.\end{theo}

\begin{proof}1) An easy computation shows that 
$\Ext^0_{\mk[y]/y^3}(\mk,\mk)=\mk[\veps]$ (dual numbers). 
By Theorem \ref{th:nilpotence_and_factorization} 2), we have an $A_{\infty}$-morphism 
$g:\mk[x]/x^6\to \End^{A_{\infty}}_{\mk[y]/y^3}(\mk),$ 
such that $\bbar{g_1(x)}=\veps\in H^0(\End^{A_{\infty}}_{\mk[y]/y^3}(\mk)).$ This gives the desired $A_{\infty}$-bimodule structure on $V.$ The rest conclusions are clear.

2) follows from 1) and Proposition \ref{prop:cat_res_implies_prop}.
\end{proof}

\begin{proof}[Proof of Theorem \ref{th:nilpotence_and_factorization}, part 1)] Let us denote by $A_f$ the $\mk[x]\mhyphen A$-bimodule which equals $A$ as an $A$-module, and whose $\mk[x]$-module structure comes from $f.$ Since the algebra $\mk[x]$ is smooth, we have $A_f\in D_{\perf}(\mk[x]\otimes A).$ Since $a\in H^0(A)$ is nilpotent, we have $\mk[x^{\pm 1}]\Ltens{\mk[x]}A=0.$ We conclude that $A_f$ is contained in the essential image of $D_{\perf}(\Lambda_1\otimes A)\hto D_{\perf}(\mk[x]\otimes A).$

Now, let us note that in $\Ho(\dgcat_{\mk})$ we have $\Perf(\Lambda_1)\simeq \colim_{n}\PsPerf(\mk[x]/x^n).$ It follows that we have an equivalence of triangulated categories $$D_{\perf}(\Lambda_1\otimes A)\simeq \colim_n D_{\perf}(\PsPerf(\mk[x]/x^n)\otimes A).$$ Therefore, there exists $n>0$ such that $A_f$ is contained in the essential image of $D_{\perf}(\PsPerf(\mk[x]/x^n)\otimes A).$ Let us denote by $\tilde{M}\in D_{\perf}(\PsPerf(\mk[x]/x^n)\otimes A)$ an object whose image is isomorphic to $A_f.$ We have a natural functor $$\Phi:\PsPerf(\mk[x]/x^n)\otimes\Perf(A)\to \bR\un{\Hom}(\mk[x]/x^n,\Perf(A)).$$
By construction, the $\mk[x]/x^n\mhyphen A$-bimodule $\Phi(\tilde{M})$ is quasi-isomorphic to $A$ as an $A$-module Choosing an isomorphism $\Phi(\tilde{M})_{\mid A}\xto{\sim} A$, we obtain the following composition morphism in $\Ho(\dgalg_{\mk}):$ 
$$g:\mk[x]/x^n\to \bR\End_A(\Phi(\tilde{M}))\xto{\sim} A.$$ By construction, $H^0(g)(x)=a.$ Thus, $g$ factors $f$ through $\mk[x]/x^n.$ This proves part 1)\end{proof}

The proof of part 2) of Theorem \ref{th:nilpotence_and_factorization} requires some computations which we split into several lemmas.

First, we may replace the abstract algebra $A$ by the concrete DG algebra $C$ which  was used in Section \ref{sec:disproving_version_for_smooth}. Recall that it is freely generated by the elements $t_1,$ $t_2$ with $|t_1|=0,$ $|t_2|=-1,$ and $dt_1=0,$ $dt_2=t_1^2.$ Indeed choosing a representative $\tilde{a}\in A^0$ of $a,$ and an element $h\in A^{-1}$ such that $dh=\tilde{a}^2,$ we obtain a morphism of DG algebras $C\to A,$ $t_1\mapsto \tilde{a},$ $t_2\mapsto h.$ Thus, we may assume that $A=C$ and $a=\bbar{t_1}.$

It will be very useful to introduce an additional $\Z$-grading on $C,$ which  can be thought of as a $\bG_m$-action. We will denote this grading by $\bw,$ putting $\bw(t_1)=1,$ $\bw(t_2)=2,$ and then extend by the rule $\bw(uv)=\bw(u)+\bw(v).$ Clearly, the differential $d$ has degree zero with respect to $w.$ We thus have a decomposition of $C$ as a complex: $C=\bigoplus\limits_{n\geq 0}C^{\bullet,n}.$ 

Let us define $\hat{C}:=\prod\limits_{n\geq 0}C^{\bullet,n}.$ This is also a DG algebra, and we have a map $C\to\hat{C}$. The homogeneous elements of degree $-m$ in $\hat{C}$ are just non-commutative power series in $t_1,t_2$ such that in each monomial there are exactly $m$ copies of $t_2.$

\begin{lemma}\label{lem:cohom_of_hat_C} The cohomology algebra $H^{\bullet}(\hat{C})$ is generated by the elements $u_1=\bbar{t_1}$ and $u_2=\bbar{[t_1,t_2]},$ with two relations: $u_1^2=0,$ $u_1u_2+u_2u_1=0.$\end{lemma}

\begin{proof}Indeed, it is easy to see that the DG algebra $\hat{C}$ is isomorphic to the endomorphism DG algebra $\End^{A_{\infty}}_{\mk[y]/y^3}(\mk).$
Thus, we have an isomorphism of graded algebras $H^{\bullet}(\hat{C})\cong \Ext^{\bullet}_{\mk[y]/y^3}(\mk,\mk).$ To compute this Ext-algebra, we take the semi-free resolution $P\to\mk.$ The underlying graded $\mk[y]/y^3$-module is defined by
$$P^{gr}:=\bigoplus_{n=0}^{\infty}e_n\cdot \mk[y]/y^3,$$
where $|e_{n}|=\lfloor \frac{n}{2}\rfloor.$ The differential is given by $d(e_0)=0,$ and $d(e_{2k+1})=e_{2k}y,$ $d(e_{2k+2})=e_{2k+1}y^2$ for $k\geq 0.$ The morphism $P\to \mk$ sends $e_0$ to $1,$ and $e_n$ to $0$ for $n>0.$ Clearly, this is a quasi-isomorphism. 

We see that $\Ext^{\bullet}_{\mk[y]/y^3}(\mk,\mk)\cong \Hom_{\mk[y]/y^3}^{\bullet}(P,\mk),$ where the last complex has zero differential, and is equipped with the homogeneous basis $\{v_n\}_{n\geq 0},$ where $|v_n|=\lfloor\frac{n}{2}\rfloor,$ and $v_i(e_j)=\delta_{ij}.$ It is easy to see that the elements $v_1$ and $v_2$ correspond to the classes $u_1,u_2\in H^{\bullet}(\hat{C}),$ mentioned in the formulation of the lemma. Clearly, we have $u_1^2=0.$ It remains to show that $u_1u_2=-u_2u_1,$ and $u_1u_2^k\ne 0$ for $k\geq 0.$ Let us choose the lifts $\wt{v_n}\in\End_{\mk[y]/y^3}(P)$ of $v_n,$ putting
$$\wt{v_{2k}}(e_n)=\begin{cases}(-1)^{nk}e_{n-2k} & \text{for }n\geq 2k,\\
0 & \text{otherwise};\end{cases}\quad \wt{v_{2k+1}}(e_n)=\begin{cases}e_{n-2k-1} & \text{for }n\text{ odd,}n\geq 2k+1,\\
(-1)^k e_{n-2k-1}y & \text{for }n\text{ even,}n\geq 2k+2,\\
0 & \text{otherwise.}\end{cases}$$
It is easy to check that $\wt{v_n}$'s super-commute with the differential, and that $\wt{v_1}\wt{v_2}+\wt{v_2}\wt{v_1}=0,$ $\wt{v}_1(\wt{v_2})^k=(-1)^k\wt{v_{2k+1}}.$ This proves the lemma.\end{proof}

\begin{lemma}\label{lem:q_iso_with_completion}The natural inclusion $C\to\hat{C}$ is a quasi-isomorphism.\end{lemma}

\begin{proof}We already know that $\dim H^n(\hat{C})<\infty$ for all $n\in\Z.$ It remains to observe the following: for any infinite sequence of complexes of vector spaces $\cK_{0}^{\bullet},\cK_1^{\bullet},\dots$ such that $\dim H^n(\prod\limits_{n\geq 0}\cK_n^{\bullet})<\infty$ for all $n\in\Z,$ the morphism $\bigoplus\limits_{n\geq 0}\cK_n^{\bullet}\to\prod\limits_{n\geq 0}\cK_n^{\bullet}$ is a quasi-isomorphism. Applying this observation to the complexes $C^{\bullet,n},$ we conclude the proof.\end{proof}

We now construct a strictly unital $A_{\infty}$-morphism $\mk[x]/x^6\to C,$ using obstruction theory. First, we introduce the weight grading ($\bG_m$-action) on $\mk[x]/x^6$ by putting $\bw(x)=1.$ Our $A_{\infty}$-morphism will be compatible with the $\bG_m$-actions, and its component $f_1$ is given by \begin{equation}\label{eq:formula_for_f_1} f_1(x^k)=t_1^k\text{ for }0\leq k\leq 5.\end{equation} Note that all the cohomology spaces $H^n(C)$ are $H^0(C)\mhyphen H^0(C)$-bimodules, hence also over $\mk[x]/x^6\mhyphen\mk[x]/x^6$-bimodules (via $f_1$).

Let us also note that for any $\mk[x]/x^6\mhyphen\mk[x]/x^6$-bimodule $M,$ equipped with the compatible $\bG_m$-action, the Hochschild cohomology $HH^{\bullet}(\mk[x]/x^6,M)$ also becomes bigraded; the second grading again corresponds to the $G_m$-action. For a vector space $V$ equipped with a $\bG_m$-action, $V=\bigoplus_{n\in Z}V^n,$ we denote $V(k)$ the same space with a twisted $\bG_m$-action: $V(k)^n=V^{k+n}.$

\begin{lemma}\label{lem:computattion_of_HH_cohom}We have $HH^{2k+2}(\mk[x]/x^6,H^{-2k}(C))\cong \mk[\veps](6)$ for $k\geq 0,$ and $HH^{2k+3}(\mk[x]/x^6,H^{-2k-1}(C))\cong \mk(4)$ ($\bG_m$-equivariant isomorphisms).\end{lemma}

\begin{proof}We have the following $\bG_m$-equivariant resolution of the diagonal bimodule:
$$\dots\xto{d_3}\mk[x]/x^6\otimes \mk[x]/x^6(-6)\xto{d_2} \mk[x]/x^6\otimes \mk[x]/x^6(-1)\xto{d_1}\mk[x]/x^6\otimes \mk[x]/x^6\xto{m}\mk[x]/x^6,$$
where $d_{2k+1}=x\otimes 1-1\otimes x,$ and $d_{2k}=x^5\otimes 1+x\otimes x^4+\dots1\otimes x^5.$ Further, by Lemmas \ref{lem:cohom_of_hat_C} and \ref{lem:q_iso_with_completion} we know the $\bG_m$-equivariant $H^0(C)\mhyphen H^0(C)$-bimodules $H^n(C).$ Namely, $H^{-2k}(C)\cong\mk[\veps](-6k)$ (twisted diagonal bimodule), and $H^{-2k-1}(C)\cong (\mk[\veps])_{\sigma}(-6k-3)$ -- the twisted anti-diagonal bimodule. For the later, the left and right $H^0(C)$-actions are given respectively by $\veps\cdot 1=\veps,$ $1\cdot \veps=-\veps.$ The result follows by an elementary computation.\end{proof}

We are finally able to finish the proof of the theorem.

\begin{proof}[Proof of Theorem \ref{th:nilpotence_and_factorization}, part 2).] As we already mentioned, we will construct (or rather show the existence of) a $\bG_m$-equivariant strictly unital $A_{\infty}$-morphism  $f:\mk[x]/x^6\to C,$ where $f_1$ is given by \eqref{eq:formula_for_f_1}. Since $H^0(f_1)$ is a homomorphism, we can construct $f_2$ such that the required relation is satisfied. Suppose that we have already constructed $\bG_m$-equivariant $f_1,\dots,f_n$ (where $n\geq 2$) satisfying all the relations for the $A_{\infty}$-morphism that involve only $f_1,\dots,f_n.$ We want to construct the $(n+1)$ components $f_1,\dots,f_n',f_{n+1}$ (again, satisfying all the relevant relations) where only $f_n$ is possibly being replaced by another map $f_n'.$ The standard obstruction theory tells us that the obstruction to this is given by a class in $HH^{n+1,0}(\mk[x]/x^6,H^{1-n}(C))$ (the $\bG_m$-invariant part). Applying Lemma \ref{lem:computattion_of_HH_cohom}, we see that this space vanishes. Thus, proceeding inductively we can construct the desired $A_{\infty}$-morphism $f.$ This proves the theorem.\end{proof}

\end{document}